\documentclass{article}
\title{Improvements to Turing's Method II.}
\author{T. S. Trudgian\footnote{Supported by ARC Grant DE120100173.}\\
Mathematical Sciences Institute\\ The Australian National University,
 ACT 0200, Australia\\ timothy.trudgian@anu.edu.au
}
\usepackage{comment}
\usepackage{url}
\usepackage{amsthm}
\usepackage{amsmath}
\usepackage{amssymb}
\newtheorem{cor}{Corollary}
\newtheorem{thm}{Theorem}
\newtheorem{lem}{Lemma}

\begin{document}

\maketitle

\begin{abstract}
\noindent
Turing's method uses explicit bounds on $|\int_{t_{1}}^{t_{2}} S(t)\, dt|$, where $\pi S(t)$ is the argument of the Riemann zeta-function. This article improves the bound on $|\int_{t_{1}}^{t_{2}} S(t)\, dt|$ given in \cite{TrudgianTuring}.
\end{abstract}

\section{Introduction}\label{intro}
Let $\zeta(s)$ be the Riemann zeta-function, and let $N(T)$ denote the number of zeroes of $\zeta(s)$ with $0<\Re(s)<1$ and $0<\Im(s)< T$. One seeks to calculate $N(T)$ as follows.

First one finds zeroes by locating sign changes of a real-valued function the zeroes of which agree with the non-trivial zeroes of the zeta-function. 
\begin{comment}
This can be done using Gram's Law, the Riemann--Siegel formula and Euler--Maclaurin summation or by some sophisticated bisection techniques and the fast Fourier transform --- see, e.g., \cite{Platt}. 
\end{comment}
This gives one a lower bound on the number of zeroes of $\zeta(s)$ with $0<\Im(s)<T$. 

To check whether this initial analysis has omitted some zeroes one employs Turing's method. This was first annunciated by Turing \cite{Turing} in 1953 and has been used extensively since then. Recently, another method has been deployed by B\"{u}the \cite{Buethe}.
 
To apply Turing's method one needs good explicit bounds on 
$$\bigg|\int_{t_{1}}^{t_{2}} S(t)\, dt\bigg|,$$
for $t_{2}>t_{1}>0$, where $\pi S(t)$ is defined to be the argument of $\zeta(\frac{1}{2} + it)$. For a complete definition and a brief history of the problem, see \cite[\S 1]{TrudgianTuring} and \cite[Ch.~7]{Edwards}.

This article improves \cite{TrudgianTuring} and contains frequent references to the results therein. The main result is
\begin{thm}\label{main}
\begin{equation}\label{postman}
\bigg| \int_{t_{1}}^{t_{2}} S(t)\, dt\bigg| \leq 1.698 + 0.183\log\log t_{2}+ 0.049\log t_{2},
\end{equation}
for $t_{2}>t_{1} > 10^{5}.$ If the right-side of (\ref{postman}) is replaced by $a + b \log\log t_{2} + c \log t_{2}$, one may use Table \ref{tabled} on page \pageref{tabled} for more specific values of $a, b$ and $c$.
\end{thm}
In \cite{TrudgianTuring} the main result followed from Lemma 2.8 and Lemma 2.11 which concerned respectively obtaining an upper and a lower bound for $\Re\log\zeta(s)$ for $\Re(s)\geq \frac{1}{2}$. This article refines only the upper bound. Theorem \ref{main} improves on Theorem 2.2 in \cite{TrudgianTuring} for all $t_{2} \geq 10^{5}$.

The idea in this article is to use more sophisticated estimates on $\zeta(\sigma + it)$ for $\frac{1}{2} \leq \sigma \leq 1$; these estimates have been given in \cite{PlattTrudgian} and \cite{TrudZeta}. A bound on $|\zeta(s)|$ is given in \S \ref{bound}, a proof of Theorem 1 is given in \S \ref{proof}, and some concluding remarks are provided in \S \ref{conc}.

\subsection*{Acknowledgements}
I should like to thank Jan B\"{u}the for detailed discussions on this problem.

\section{Bounding $|\zeta(\sigma + it)|$ across the strip $\frac{1}{2}\leq\sigma\leq 1+ \delta$}\label{bound}

Using the inequality $\log (1+ x) \leq x$ it is easy to see that 
\begin{equation}\label{eq1}
\log|Q_{0} + \sigma + it| - \log t \leq \frac{1}{2} \left( \frac{\sigma_{1} + Q_{0}}{t_{0}}\right)^{2},
\end{equation}
for $\sigma \leq \sigma_{1}$ and $t\geq t_{0}$ and any $Q_{0}\geq 0$. With the trivial observations $\log|Q_{0} + \sigma + it| \geq \log t$ and $|\arg (Q_{0} + \sigma + it)| \leq \frac{\pi}{2}$ at hand, we may apply (\ref{eq1}) to see that
\begin{equation}\label{zero}
|\log (Q_{0} + \sigma + it)| \leq (1+ a_{0}) \log t, \quad\quad (\sigma \leq \sigma_{1})
\end{equation}
where
\begin{equation*}\label{a0def}
a_{0} = a_{0}(\sigma_{1}) = \frac{\sigma_{1} + Q_{0}}{2 t_{0}^{2} \log t_{0}} + \frac{\pi}{2\log t_{0}} + \frac{\pi(\sigma_{1} + Q_{0})^{2}}{4 t_{0} \log^{2} t_{0}}.
\end{equation*}

Suppose that
\begin{equation}\label{int}
|\zeta(\tfrac{1}{2} + it)| \leq k_{1} t^{k_{2}} (\log t)^{k_{3}}, \quad (t\geq t_{1}), \quad\quad |\zeta(1+ it)|\leq k_{4} \log t^{k_{5}}, \quad (t\geq t_{2}).
\end{equation}
Consider the function $h(s) = (s-1) \zeta(s)$, which is entire. Once we are able to exhibit bounds for $|h(s)|$ using the information in (\ref{int}) we can apply a version of the Phragm\'{e}n--Lindel\"{o}f principle to bound $|\zeta(s)|$.
Using Lemma 3 in \cite{TrudgianS2} and (\ref{zero}) and (\ref{eq1}) we may prove
\begin{lem}\label{lem1}
Let $h(s) = (s-1) \zeta(s)$, and let $\delta$ be a positive real number. Furthermore, let $Q_{0}\geq 0$ be a number for which
\begin{equation*}\label{one}
\begin{split}
|h(\tfrac{1}{2} +it)| &\leq k_{1} |Q_{0} + \tfrac{1}{2} + it|^{k_{2} + 1} (\log| Q_{0} + \tfrac{1}{2} + it|)^{k_{3}}\\
|h(1 + it)& \leq k_{4}|Q_{0} + 1+ it| (\log|Q_{0} + 1+ it|)^{k_{5}} \\
|h(1 + \delta + it)| &\leq \zeta(1+ \delta)|Q_{0} + 1+ \delta + it|,
\end{split}
\end{equation*}
for all $t$.
Then for $\sigma \in[\frac{1}{2}, 1]$ and $t\geq t_{0}$,
\begin{equation}\label{p1}
 |\zeta(s)| \leq \alpha_{1} k_{1}^{2(1-\sigma)} k_{4}^{2(\sigma- \frac{1}{2})} t^{2k_{2}(1-\sigma)} (\log t)^{2(k_{3}(1-\sigma) + k_{5}(\sigma - \frac{1}{2}))},
 \end{equation}
where 
\begin{equation*}\label{ce}
\alpha_{1} = (1+ a_{1}(1+ \delta, Q_{0}, t_{0}))^{k_{2} + 1} (1+a_{0}(1+ \delta, Q_{0}, t_{0}))^{k_{3} + k_{5}},
\end{equation*}
and
\begin{equation*} 
a_{0}(\sigma, Q_{0}, t) = \frac{\sigma + Q_{0}}{2 t^{2} \log t} + \frac{\pi}{2\log t} + \frac{\pi(\sigma + Q_{0})^{2}}{4 t \log^{2} t}, \quad a_{1}(\sigma, Q_{0}, t) = \frac{\sigma + Q_{0}}{t}.
\end{equation*}
Whereas for $\sigma\in[1, 1+ \delta]$ and $t\geq t_{0}$,
\begin{equation}\label{p2}
|\zeta(s)| \leq \alpha_{2} k_{4}^{\frac{1+ \delta - \sigma}{\delta}} \zeta(1+ \delta)^{\frac{\sigma -1}{\delta}} (\log t)^{k_{5}(\frac{1+ \delta -\sigma}{\delta})},
\end{equation}
where 
\begin{equation*}\label{a3}
\alpha_{2} = (1+ a_{1}(1+ \delta, Q_{0}, t_{0}))(1 + a_{0}(1+ \delta, Q_{0}, t_{0}))^{k_{5}}.
\end{equation*}
Finally, for all $\sigma\in[\frac{1}{2}, 1+ \delta]$ and $t\geq t_{0}$ we have
\begin{equation*}\label{kite}
|\zeta(s)| \leq \left(1+ a_{1}(1+ \delta, Q_{0})\right)^{k_{2} + 1}\left(1+ a_{0}(1+ \delta, Q_{0})\right)^{k_{3} + k_{5}}k_{1} t^{k_{2}} (\log t)^{k_{3}},
\end{equation*}
provided that
\begin{equation}\label{conditions}
t^{k_{2}} (\log t)^{k_{3} - k_{5}} \geq \frac{k_{4}}{k_{1}}, \quad
 t \geq \exp\left\{ \left(\frac{\zeta(1+ \delta)}{k_{4}}\right)^{\frac{1}{k_{5}}}\right\}.
\end{equation}
\end{lem}
\begin{proof}
In applying Lemma 3 of \cite{TrudgianS2} to $h(s)$ we need to relate $|Q_{0} + s|$ and $|s-1|$ to $t$. We simply note that
$$ \bigg| \frac{Q_{0} + s}{s-1}\bigg| \leq \frac{|Q_{0} + s|}{t} \leq 1+ \frac{\sigma + Q_{0}}{t_{0}} = 1 + a_{1}(\sigma, Q_{0}) \leq 1+ a_{1}(1+ \delta, Q_{0}),$$
in both regions $\sigma \in [\frac{1}{2}, 1]$ and $\sigma \in[1, 1+ \delta]$. %Though the arguments of $a_{1}$ and $a_{0}$ in (\ref{ce}) are both $\sigma = 1$, we include the coarser bound $\sigma = 1+ \delta$ for ease of exposition in \S \ref{proof}. 
Since $a_{0}$ and $a_{1}$ are small for any respectable value of $t_{0}$ we throw away some information in the exponents of $1+ a_{1}$ and $1+ a_{0}$. For example, in proving (\ref{p1}) we arrive at
$$(1+ a_{1})^{2k_{2}(1- \sigma) +1} (1+ a_{0})^{2k_{3}(1-\sigma) + 2k_{5}(\sigma - \frac{1}{2})}.$$
Rather than retain this dependence on $\sigma$ in the exponents, we simply bound $1-\sigma$  and $\sigma - \frac{1}{2}$ by $\frac{1}{2}$. A similar procedure is applied to prove (\ref{p2}).

To prove the bound in the region $\sigma\in[\frac{1}{2} , 1+ \delta]$ we note that the bounds in (\ref{p1}) and (\ref{p2}) are decreasing in $\sigma$ if the inequalities in (\ref{conditions}) are met. Finally, the bound in (\ref{p1}), evaluated at $\sigma = \frac{1}{2}$ exceeds the bound in (\ref{p2}), evaluated at $\sigma = 1$. This completes the lemma.
\end{proof}
It is worth recording values of $k_{1}, \ldots, k_{5}$, which we do in
\begin{cor}\label{jeep}
For $\sigma\in[\frac{1}{2}, 1+ \delta]$ and $t\geq t_{0}$ we have
\begin{equation*}\label{zephyr}
|\zeta(s)| \leq 0.732(1 + a_{1}(1+ \delta, 5, t_{0}))^{7/6} (1+ a_{0}(1+ \delta,5, t_{0}))^{2} t^{1/6} \log t,
\end{equation*}
provided that
\begin{equation*}\label{boreas}
t\geq \max\{ 1.16, \exp[4\zeta(1+ \delta)/3]\}.
\end{equation*}
\end{cor}
\begin{proof}
In \cite{TrudZeta} it was shown that $|\zeta(1+ it)| \leq \frac{3}{4} \log t$ for $t\geq 3$. In \cite{PlattTrudgian} it was shown that 
$|\zeta(\frac{1}{2} + it)| \leq 0.732 |4.678+it|^{\frac{1}{6}} \log |4.678+it|$ for all $t$. One may therefore choose 
\begin{equation}\label{values}
(k_{1}, k_{2}, k_{3}, k_{4}, k_{5}, Q_{0}) = (0.732, \tfrac{1}{6}, 1, \tfrac{3}{4}, 1, 5)
\end{equation}
 in Lemma \ref{lem1}, which proves the corollary.
\end{proof}
Although we shall use (\ref{values}) in our computation  we proceed with the variables $(k_{1}, \ldots, Q_{0})$ as parameters. 
We remark that, although Corollary \ref{jeep} is not used in this article, it is derived at very little additional cost and should prove useful for related problems.
%see \cite{1over}.

\section{Proof of Theorem \ref{main}}\label{proof}
We are now able to proceed to the proof of Theorem \ref{main}. We need to give an upper bound for
$$I = \Re\int_{\frac{1}{2} + it}^{\infty + it} \log \zeta(s)\, ds.$$
To that end, we shall write
\begin{equation}\label{L2}
I \leq \int_{\frac{1}{2} + it}^{1+ it} \log |\zeta(s)|\, ds  +  \int_{1 + it}^{1+ \delta+ it} \log |\zeta(s)|\, ds + \int_{1+ \delta}^{\infty} \log |\zeta(\sigma)| \, d\sigma,
\end{equation}
and apply (\ref{p1}) to the first integral in (\ref{L2}) and (\ref{p2}) to the second. This gives us
\begin{lem}\label{second}
For $t\geq t_{0}$ we have
\begin{equation*}
I\leq A_{1} + B_{1} \log\log t + C_{1} \log t,
\end{equation*}
where
\begin{equation*}
\begin{split}
A_{1} &= \int_{1+\delta}^{\infty} \log |\zeta(\sigma)| \, d\sigma 
+ (\frac{k_{2}}{4} + \frac{1}{2} + \delta) \log (1+ a_{1}(1+ \delta, Q_{0}, t_{0})) 
+ \frac{1}{4} \log k_{1} \\
&
+ (\frac{k_{3}}{4} + \frac{k_{5}}{4} + \frac{k_{5}\delta}{2}) \log (1+ a_{0}(1+ \delta, Q_{0}, t_{0}))
+ (\frac{1}{4} + \frac{\delta}{2})\log k_{4} 
+ \frac{\delta}{2} \log \zeta(1+ \delta)
\end{split}
\end{equation*}
and
\begin{equation*}
B_{1} = \left( \frac{k_{3} + k_{5}}{4} + \frac{\delta k_{5}}{2}\right), \quad C_{1} = \frac{k_{2}}{4}.
\end{equation*}
\end{lem}
The term corresponding to $C_{1}$ in Lemma 2.8 of \cite{TrudgianTuring} is, $k_{2}/4+ \delta k_{2}/2$. Since we are not permitted to take $\delta$ too small, lest the integral in $A_{1}$ become too large, this represents a considerable qualitative saving. This is due entirely to estimating $|\zeta(s)|$, not in one go over $\sigma\in[\frac{1}{2}, 1+ \delta]$ as in \cite{TrudgianTuring} but by using Lemma \ref{lem1}.
We combine Lemma \ref{second} with Lemma 2.11 in \cite{TrudgianTuring} to obtain
\begin{thm}\label{T2}
For $t_{2}\geq t_{1} \geq 10^{5}$,
\begin{equation*}
\bigg| \int_{t_{1}}^{t_{2}} S(t)\, dt\bigg| \leq a + b\log\log t_{2}+ c\log t_{2},
\end{equation*}
where 
\begin{equation}\label{fina}
\begin{split}
\pi a &= \int_{1+ \delta}^{\infty} \log |\zeta(\sigma)| \, d\sigma 
+ (\frac{k_{2}}{4} + \frac{1}{2} + \delta) \log (1+ a_{1}(1+ \delta, Q_{0}, 10^{5})) + \frac{1}{4} \log k_{1}  \\
&
+ (\frac{k_{3}}{4} + \frac{k_{5}}{4} + \frac{k_{5}\delta}{2}) \log (1+ a_{0}(1+ \delta, Q_{0}, 10^{5}))
+ (\frac{1}{4} + \frac{\delta}{2})\log k_{4} 
+ \frac{\delta}{2} \log \zeta(1+ \delta)\\
 &+ d^{2} \log 4\left\{ -\frac{\zeta'(\frac{1}{2} + d)}{\zeta(\frac{1}{2} + d)} - \frac{1}{2} \log 2\pi + \frac{1}{4}\right\} + \frac{d^{2}}{2} \log \pi - \frac{1}{2} \int_{1+ 2d}^{\infty} \log \zeta(\sigma)\, d\sigma \\
 & + \int_{\frac{1}{2} + d}^{\infty}\log \zeta(\sigma)\, d\sigma - \frac{1}{2} \int_{1+ 2d}^{1+ 4d} \log \zeta(\sigma)\, d\sigma 
+ \int_{\frac{1}{2} + d}^{\frac{1}{2} + 2d} \log \zeta(\sigma)\, d\sigma + 3\times 10^{-4},
\end{split}
\end{equation}
and 
\begin{equation}\label{finb}
\pi b = \left( \frac{k_{3} + k_{5}}{4} + \frac{\delta k_{5}}{2}\right), \quad\quad \pi c = \frac{k_{2}}{4} + \frac{d^{2}}{2}(\log 4 -1).
\end{equation}
\end{thm}
\subsection{Computation}
Before we commence an analysis of the coefficients appearing in Theorem \ref{T2} we make the following observation. One may replace the values of $(k_{1}, k_{2}, k_{3})$ in (\ref{values}) by
\begin{equation}\label{bounds}
(k_{1}, k_{2}, k_{3}) = \left(\frac{4}{(2\pi)^{\frac{1}{4}}}, \frac{1}{4}, 0\right),
\end{equation}
which appear in \cite[Lem.\ 2]{Lehman}. The values in (\ref{bounds}) are obtained using the approximate functional equation of $\zeta(s)$: the values in (\ref{values}) are obtained using exponential sums. The value $k_{2}= \frac{1}{4}$ follows from convexity theorems. We call  (\ref{bounds}) the convexity result, and (\ref{values}) the sub-convexity result. 

As in Theorem 2.12 in \cite{TrudgianTuring} no term in either (\ref{fina}) or (\ref{finb}) depends on both $\delta$ and $d$. We can run two one-dimensional optimisations on each of $a$, $b$ and $c$. In Table \ref{tabled} we compare the results obtained from the convexity result (C), the sub-convexity result (SC), and the coefficients in Theorem 2.2 of \cite{TrudgianTuring}, when $t_{1} \approx T$. The values of $\delta, d, a, b$ and $c$ correspond to the sub-convexity result. We find that the sub-convexity result overtakes the convexity result when $T\geq 2.85\times 10^{10},$ which is, just barely, beneath the height to which the Riemann hypothesis has been verified --- see \cite{Plattarxiv}.

The values used in Theorem \ref{main} are taken from the row $T = 10^{10}$. It should be stressed that all of the results in this table are valid for $t_{1}\geq 10^{5}$. The stated values of $a, b$ and $c$  are those that are close to the best values obtainable by this method when $t_{1} \approx T$.

\begin{table}[ht]
\caption{Comparison of bounds for $|\int_{t_{1}}^{t_{2}} S(t)\, dt|\leq a + b \log\log t_{2} + c \log t_{2}$}
\centering
\begin{tabular}
{c c c c c c c c c}
\hline\hline
$T$ & Thm 2.2 & C & SC & $d$ & $\delta$ & $a$ & $b$ & $c$  \\[0.5ex]\hline
$10^{5}$ & 2.747 & 2.629 & 2.658 & 0.883 & 0.279 & 1.457 & 0.204 & 0.062 \\
$10^{6}$ & 2.883 & 2.800 & 2.827 & 0.845 & 0.237 & 1.520 & 0.197 & 0.058 \\
$10^{7}$ & 3.018 & 2.959 & 2.982& 0.817 & 0.206 & 1.573 & 0.192 & 0.055 \\
$10^{8}$ & 3.154 & 3.110 & 3.128& 0.795 & 0.182 & 1.620 & 0.189 & 0.053 \\
$10^{9}$ & 3.290 & 3.255 & 3.266& 0.777 & 0.163 & 1.661 & 0.186 & 0.051  \\
$10^{10}$ & 3.426 & 3.395 & 3.398& 0.762 & 0.148 & 1.698 & 0.183 & 0.049  \\
$10^{11}$ & 3.562 & 3.530 & 3.526 & 0.749 & 0.135 & 1.733 & 0.181 & 0.048 \\
$10^{12}$ & 3.698 & 3.663 & 3.649 & 0.738 & 0.124 & 1.764 & 0.179 & 0.047 \\
$10^{13}$ & 3.834 & 3.792 & 3.770 & 0.729 & 0.115 & 1.792 & 0.178 & 0.046 \\
$10^{14}$ & 3.969 & 3.919 & 3.887 & 0.720 & 0.107 & 1.820 & 0.177 & 0.046 \\
$10^{15}$ & 4.105 & 4.044 & 4.002& 0.713 & 0.100 & 1.844 & 0.176 & 0.045  \\
 \hline\hline
\end{tabular}
\label{tabled}
\end{table}

\section{Conclusion}\label{conc}
It seems difficult to improve substantially on Theorem \ref{main}. 
%As is noted in \cite[\S 5]{PlattTrudgian} the value of $k_{1} = 0.732$ appears to be far from being the optimal one. 
Given that the improvements obtained in this paper are only modest, and since further improvements would require a lot of effort in estimating $\zeta(\frac{1}{2} +it)$ or $\zeta(1+ it)$, it seems hopeless to try to improve this part of the argument. 
\begin{comment}
To illustrate the point further, suppose the Lindel\"{o}f hypothesis were true and one had $k_{2} = 0$. Keeping all other constants the same we would then have, for example $|\int_{t_{1}}^{t_{2}}S(t)\, dt|\leq 3.093$ where $T\approx 10^{15}$.
\end{comment}

One could try one's luck at reducing the term $\log 4$ that appears in both (\ref{fina}) and (\ref{finb}). This comes from Lemma 4.4 in \cite{Booker}.  Reducing this would have a more profound influence on bounding $|\int_{t_{1}}^{t_{2}}S(t)\, dt|$ than better bounds for $|\zeta(s)|$.

Finally, it is worth considering Theorems 3.3 and 4.3 in \cite{TrudgianTuring}, which relate to Dirichlet $L$-functions and Dedekind zeta-functions. Both of these could be improved, in line with this article, were one in possession of explicit estimates on the lines $\sigma = \frac{1}{2}$ and $\sigma = 1$. 
One such estimate, bounding $|L(1+ it, \chi)|$ for $L(s)$ a Dirichlet $L$-function, appears in \cite{DudekL}. It is possible that this could be used to obtain an improvement to Theorem 3.3 in \cite{TrudgianTuring}.

\bibliographystyle{plain}
\bibliography{themastercanada}

\end{document}